\documentclass[10pt,A4paper]{amsart}
\usepackage{amssymb,amsmath,amsthm,hyperref}
\usepackage[mathscr]{eucal}

\DeclareMathOperator{\Aut}{Aut}  

\DeclareMathOperator{\rk}{rk}

\newcommand{\Fraisse}{Fra\"iss\'e}

\DeclareMathOperator{\td}{td}  
\DeclareMathOperator{\loc}{Loc}   

\DeclareMathOperator{\Mat}{Mat}  

\DeclareMathOperator{\ldim}{ldim}  

\DeclareMathOperator{\ecl}{ecl} 

\newcommand{\restrict}[1]{\ensuremath{\!\!\upharpoonright_{#1}}}

\newcommand{\N}{\ensuremath{\mathbb{N}}}
\newcommand{\Z}{\ensuremath{\mathbb{Z}}}
\newcommand{\Q}{\ensuremath{\mathbb{Q}}}

\newcommand{\C}{\ensuremath{\mathbb{C}}}

\newcommand{\Cexp}{\ensuremath{\mathbb{C}_{\mathrm{exp}}}}

\newcommand{\Loo}{\ensuremath{L_{\omega_1,\omega}}}

\newcommand{\ACF}{\ensuremath{\mathrm{ACF}}}

\newcommand{\ga}{\ensuremath{\mathbb{G}_\mathrm{a}}}   
\newcommand{\gm}{\ensuremath{\mathbb{G}_\mathrm{m}}}  

\renewcommand{\phi}{\varphi}
\renewcommand{\le}{\ensuremath{\leqslant}}
\renewcommand{\ge}{\ensuremath{\geqslant}}
\newcommand{\tuple}[1]{\ensuremath{\langle #1 \rangle}}
\newcommand{\class}[2]{\ensuremath{\left\{ #1 \,\left|\, #2 \right.\right\}}}

\newcommand{\onto}{\twoheadrightarrow}

\newcommand{\subs}{\subseteq} 

\newcommand{\minus}{\ensuremath{\smallsetminus}}

\newcommand{\gen}[1]{\ensuremath{\left\langle #1 \right\rangle}} 

\newcommand{\cross}{\ensuremath{\times}}

\newcommand{\leteq}{\mathrel{\mathop:}=}

\DeclareMathOperator{\Gcl}{\Gamma cl}  

\newcommand{\Fexp}{\ensuremath{F_\mathrm{exp}}}

\newcommand{\B}{\ensuremath{\mathbb{B}}}
\newcommand{\Bexp}{\ensuremath{\B_\mathrm{exp}}}

\newcommand{\Fdiff}{\ensuremath{F_{\mathrm{diff}}}}

\newcommand{\h}{\ensuremath{\mathscr{H}}}
\newcommand{\G}{\ensuremath{\mathcal{G}}}

\newcommand{\DCF}{\ensuremath{\mathrm{DCF}}}

\newcommand{\TEDE}{\ensuremath{T_{\mathrm{EDE}}}}
\newcommand{\FEDE}{\ensuremath{F_{\mathrm{EDE}}}}

\newcommand{\GDE}{\ensuremath{\Gamma_{\mathrm{DE}}}}

\newcommand{\GAE}{\ensuremath{\Gamma_{\mathrm{AE}}}}    
\newcommand{\CAE}{\ensuremath{\C_{\mathrm{AE}}}}    
\newcommand{\BAE}{\ensuremath{\B_{\mathrm{AE}}}}    

\newcommand{\GBE}{\ensuremath{\Gamma_{\mathrm{BE}}}}    
\newcommand{\CBE}{\ensuremath{\C_{\mathrm{BE}}}}    
\newcommand{\BBE}{\ensuremath{\B_{\mathrm{BE}}}}    

\newtheorem{prop}{Proposition}[section]
\newtheorem{cor}[prop]{Corollary}
\newtheorem{theorem}[prop]{Theorem}
\newtheorem{lemma}[prop]{Lemma}
\newtheorem{conj}[prop]{Conjecture}
\newtheorem{fact}[prop]{Fact}

\theoremstyle{definition}
\newtheorem{defn}[prop]{Definition}

\newtheorem{remark}[prop]{Remark}

\title{Blurred complex exponentiation}

\author{Jonathan Kirby}

\date{23 October 2019} 

\address{Jonathan Kirby, School of Mathematics, University of East Anglia, Norwich Research Park, Norwich NR4 7TJ, UK}
\email{jonathan.kirby@uea.ac.uk}


\begin{document}

\begin{abstract}
It is shown that the complex field equipped with the \emph{approximate exponential map}, defined up to ambiguity from a small group, is quasiminimal: every automorphism-invariant subset of $\C$ is countable or co-countable. If the ambiguity is taken to be from a subfield analogous to a field of constants then the resulting \emph{blurred exponential field} is isomorphic to the result of an equivalent blurring of Zilber's exponential field, and to a suitable reduct of a differentially closed field. These results are progress towards Zilber's conjecture that the complex exponential field itself is quasiminimal. A key ingredient in the proofs is to prove the analogue of the exponential-algebraic closedness property using the density of the group governing the ambiguity with respect to the complex topology.

\end{abstract}

\maketitle

\section{Introduction}

This paper makes progress towards Zilber's quasiminimality conjectures for the complex exponential field.

\begin{conj}[Zilber's weak quasiminimality conjecture, \cite{Zilber97}]
The complex exponential field $\Cexp = \tuple{\C;+,\cdot,\exp}$ is quasiminimal.
\end{conj}
Zilber defined a structure to be quasiminimal if every definable subset (in one free variable) is countable or co-countable. Since logics stronger than first-order logic are used in his work, there is an ambiguity in what definable should mean. In this paper the strongest reasonable version of the definition is used.
\begin{defn}
A structure $M$ is \emph{quasiminimal} if for every countable subset $A \subs M$ (of ``parameters'') if $S \subs M$ is invariant under $\Aut(M/A)$ then $S$ is countable or its complement is countable.
\end{defn}
Zilber constructed a quasiminimal exponential field $\Bexp$ which gives rise to a stronger form of the quasiminimality conjecture \cite{Zilber00fwpe}, \cite{Zilber05peACF0}.
\begin{conj}[Zilber's strong quasiminimality conjecture]
The exponential fields $\Cexp$ and $\Bexp$ are isomorphic.
\end{conj}
Since $\Bexp$ is only defined up to isomorphism, the strong conjecture is more accurately stated in the form that $\Cexp$ satisfies the axioms that specify $\Bexp$ up to isomorphism. The key axioms take the form of Schanuel's conjecture of transcendental number theory, and a complementary property called \emph{strong exponential-algebraic closedness}.

The paper \cite{TEDESV} studied the solution set of the exponential differential equation, and showed it could be axiomatized in a very similar way to $\Bexp$. For a differential field $\Fdiff = \tuple{F;+,\cdot,D}$, let $\FEDE$ be the \emph{exponential differential equation} reduct $\FEDE = \tuple{F;+,\cdot,\GDE}$ where $\GDE \subs F^2$ is defined by
\[Dy = y Dx \,\wedge\, y\neq 0.\]

Both the graph $\G$ of the exponential function and the solution set $\GDE$ are subgroups of the algebraic group $G = \ga \cross \gm$, where $\ga$ is the additive group of the field and $\gm$ is the multiplicative group. If $F$ is a differential field of complex functions where exponentiation makes sense, then $(x,y) \in \GDE(F)$ if and only if $(\exists c \in \C)[y = e^{x+c}]$, so $\GDE$ is the graph $\G$ of the exponential function up to a blurring by the field of constants.

The main idea of this paper is that we can similarly blur the exponentiation in $\Cexp$ and $\Bexp$ with respect to suitable subfields or even subgroups, and then prove quasiminimality results about them. We consider both a coarse blurring and a finer blurring. For the coarse blurring we define $\GBE$ (blurred exponentiation) by $(x,y) \in \GBE$ if and only if $(\exists c \in C)[y = e^{x+c}]$, where $C$ is a countable subfield of $\C$ or of $\B$ which can be treated analogously to a field of constants. Specifically, $C$ is $\ecl(\emptyset)$, the subfield of exponentially algebraic numbers. We can prove an analogue of the strong quasiminimality conjecture.
\begin{theorem}\label{blurred exp theorem}\
\begin{enumerate}
\item The blurred exponential fields $\CBE = \tuple{\C;+,\cdot,\GBE}$ and $\BBE = \tuple{\B;+,\cdot,\GBE}$ are isomorphic, and quasiminimal. 
\item For an appropriate differentially closed field $\Fdiff$, the reduct $\FEDE$ is isomorphic to them. 
\item The common first-order theory of these structures is $\aleph_0$-stable. 
\item An isomorphic copy of these structures can be constructed by the amalgamation-with-predimension and excellence techniques.
\end{enumerate}
\end{theorem}
The blurring process preserves the exponential algebraic closure pregeometry, $\ecl$. Thus we can deduce:
\begin{theorem}\label{pregeom theorem}
The fields with pregeometries $\tuple{\C;+,\cdot,\ecl^\C}$ and $\tuple{\B;+,\cdot,\ecl^\B}$ are isomorphic.
\end{theorem}
This theorem can be viewed as establishing the ``geometric part'' of the strong quasiminimality conjecture.

\medskip

We also consider a relatively fine blurring of \Cexp, which we call \emph{approximate exponentiation}, although the approximation is up to a group of \Q-rank 2, not a metric approximation. The group $\GAE$ is given by $(x,y )\in \GAE$ if and only if $(\exists c \in \Q + 2\pi i \Q) [y= e^{x+c}]$. In this setting we prove the analogue of the weak quasiminimality conjecture.
\begin{theorem}\label{CAE is qm}
The approximate exponential field $\CAE \leteq \tuple{\C;+,\cdot,\GAE}$ is quasiminimal.
\end{theorem}

The main new idea in the proofs of the theorems is to use the density of the subgroup $\Q + 2\pi i \Q$ (or $\ga(C)$) in $\C$ with respect to the complex topology to prove that appropriate algebraic subvarieties of $G^n$ meet $\GAE^n$ (or $\GBE^n$). This is enough to prove the $\Gamma$-closedness property, which is the equivalent of the exponential-algebraic closedness property satisfied by \Bexp. The other ideas are taken or adapted from \cite{TEDESV} and from \cite{PEM}.

This paper has had a long gestation. Theorems~\ref{blurred exp theorem} and~\ref{pregeom theorem} were originally conceived during the Model Theory semester at the Isaac Newton Institute, Cambridge, in 2005. However there was a gap in the argument as although I could show that $\CBE$ was $\Gamma$-closed, I needed to show it was \emph{strongly $\Gamma$-closed} in order to get the isomorphism with $\BBE$. This gap was finally filled in joint work with Martin Bays for the paper \cite{PEM} where we proved that $\Gamma$-closedness implies generic strong $\Gamma$-closedness, which in the setting of $\CBE$ is equivalent to strong $\Gamma$-closedness. We also showed that generic strong $\Gamma$-closedness (together with the countable closure property) implies quasiminimality, and this fact is used for Theorem~\ref{CAE is qm}.

\subsection*{Overview}
In Section~\ref{EDE section}, the theory \TEDE\ of the exponential differential equation is recalled from \cite{TEDESV}. Section~\ref{G-field section} introduces $\Gamma$-fields and explains how blurred exponential fields are examples of them, before giving details of the pregeometry on a $\Gamma$-field based on the Ax-Schanuel property. Section~\ref{G-closed section} explains various notions of the $\Gamma$-closedness property, and Section~\ref{QM section} uses the equivalence of some of these notions to prove that a strengthening $\TEDE^*$ of $\TEDE$ is uncountably categorical, and its models are quasiminimal. Section~\ref{density section} contains the proof of the $\Gamma$-closedness property using density of the blurring group in the complex topology, and Section~\ref{proofs section} pulls everything together to give proofs of the main theorems. The final Section~\ref{remarks section} consists of some remarks and questions about the quasiminimality of $\Cexp$ itself and of other related structures.

\subsection*{Acknowledgements}
I would like to thank Boris Zilber for many useful conversations. 
I learned the idea of blurring a model-theoretically wild structure to produce a stable structure
from his paper \cite{Zilber04}.

 I would also like to thank the many seminar audiences in the American Midwest who made helpful comments when I talked about these ideas in 2006/07.

\section{The exponential differential equation}\label{EDE section}

Let $F_\mathrm{diff} = \tuple{F;+,\cdot,D}$ be a differentially closed field of characteristic zero.

\begin{defn}
We define $\GDE \subs F^2$ to be the solution set of the exponential differential equation. That is,
\[\GDE \leteq \class{(x,y) \in \ga(F)\times\gm(F)}{Dy = y Dx\, \wedge\, y \neq 0}.\]
We write $\FEDE$ for the reduct $\tuple{F;+,\cdot,\GDE}$ of $F_{\mathrm{diff}}$.
\end{defn}
Observe that the subfield $C$ of constants of the differential field is defined in $\FEDE$ by $(x,1) \in \GDE$. So we can add $C$ to the language without changing the definable sets. Indeed throughout this paper the language for a structure only matters up to which sets are $\emptyset$-definable. So the notion of \emph{reduct} is always meant in the sense of $\emptyset$-definable sets.

\medskip

We write $G$ for the algebraic group $\ga \cross \gm$. It is easy to see that $\GDE$ is a subgroup of $G(F)$, and hence a $\Z$-submodule.  Since $G$ is a $\Z$-module, $G^n$ is naturally a module over the ring $\Mat_n(\Z)$ of $n\times n$ integer matrices. The axiomatization of $\FEDE$ involves this module structure as we now explain.

The algebraic group $G$ is isomorphic to the tangent bundle $T\gm$ of $\gm$, so we can identify $G^n$ with $T\gm^n$.
The algebraic subgroups of $\gm^n$ are given by conjunctions of equations of the form $\prod_{i=1}^n y_i^{m_i} = 1$, with the $m_i \in \Z$. In matrix form we can abbreviate such conjunctions of equations as $y^M = 1$. If $J$ is the algebraic subgroup of $\gm^n$ given by $y^M = 1$ then we write $TJ$ for the algebraic subgroup of $G^n$ given by $Mx = 0$ and $y^M = 1$, where $x = (x_1,\ldots,x_n)$ are the coordinates on $\ga^n$ and $Mx$ denotes the usual matrix multiplication.

Given $M \in \Mat_n(\Z)$, and a subvariety $V \subs G^n$, we write $M\cdot V$ for the image of $V$ under $M$. The action of $M$ is by regular maps, so $M\cdot V$ is a constructible subset of $G^n$, so has a dimension.
\begin{defn}
An irreducible subvariety $V$ of $G^n$ is \emph{rotund} if for all $M \in \Mat_n(\Z)$ we have $\dim M\cdot V \ge \rk M$, the rank of the matrix.

$V$ is \emph{strongly rotund} if for all such $M$ we have $\dim M\cdot V > \rk M$, except when $M = 0$ when necessarily $\dim M\cdot V = 0$.

\end{defn}
In particular, taking $M$ to be the identity matrix, if $V \subs G^n$ is rotund then $\dim V \ge n$.

\begin{fact}\label{TEDE axioms}
The following properties are first-order expressible and axiomatize the complete first-order theory $\TEDE$ of $\tuple{F;+,
\cdot,\GDE}$. For the axioms we will use the symbol $\Gamma$ instead of $\GDE$.
  \begin{description}
  \item[$\mathbf{ACF_0}$] $\tuple{F;+,\cdot}$ is an algebraically closed
    field of characteristic zero.
  \item[\textup{Group}] $\Gamma$ is a subgroup of $\ga(F)\cross \gm(F)$.
  \item[\textup{Fibres}] The set $\class{x \in F}{(x,1)\in \Gamma}$ is an algebraically closed subfield of $F$, denoted by $C$, and
  $(0, y) \in \Gamma \iff y \in \gm(C)$.
  \item[Ax-Schanuel property] If $a \in \Gamma^n$ satisfies $\td(a/C) < n+1$ then
    there is a proper algebraic subgroup $J$ of $\gm^n$ such that
    $a \in TJ + G^n(C)$. (Here we write the group operation additively.)
  \item[\textup{Full $\Gamma$-closedness}] For each $n \in
    \N$, and each rotund irreducible algebraic subvariety $V$ of
    $G^n$, the intersection $\Gamma^n \cap V$ is
    nonempty.
\item[Non-triviality] $\exists x[x \notin C]$
  \end{description}
\end{fact}
\begin{proof}
Theorem~4.11 of the paper \cite{TEDESV} gives this theorem in more generality, for the exponential differential equations of a set $\mathcal S$ of semiabelian varieties. The specific case we use is when $\mathcal S = \class{\gm^n}{n \in \N}$.
\end{proof}
\begin{remark}
As a reduct of $\DCF_0$, the theory $\TEDE$ is $\aleph_0$-stable and has Morley rank at most $\omega$. Since it is an expansion of the pair of algebraically closed fields $\tuple{F;+,\cdot,C}$, it has Morley rank exactly $\omega$. This simple observation proves part (3) of Theorem~\ref{blurred exp theorem}, assuming parts (1) and (2).
\end{remark}

\section{$\Gamma$-fields}\label{G-field section}

\begin{defn}
An \emph{exponential field} $\Fexp = \tuple{F;+,\cdot,\exp_F}$ is a field of characteristic zero equipped with a group homomorphism $\exp_F: \ga(F) \to \gm(F)$.
\end{defn}
In this paper we will consider only the complex exponential field $\Cexp$ and Zilber's exponential field $\Bexp$ (see Definition~\ref{B defn}).

To fit exponential fields and the exponential differential equation into one framework, we use the notion of a $\Gamma$-field, which was introduced in \cite{PEM}. We only consider $\Gamma$-fields with respect to the algebraic group $G = \ga \cross \gm$. In \cite{PEM} a more general definition was given.
\begin{defn}
A \emph{$\Gamma$-field} $F_\Gamma = \tuple{F;+,\cdot,\Gamma}$ consists of a field $F$ of characteristic 0, and a divisible subgroup $\Gamma \subs G(F)$. 

A $\Gamma$-field $F_\Gamma$ is \emph{full} if $F$ is algebraically closed,  $\pi_1(\Gamma) = \ga(F)$, and $\pi_2(\Gamma) = \gm(F)$, where $\pi_1:G \onto \ga$ and $\pi_2: G\onto \gm$ are the projection maps.
\end{defn}
An exponential field is a $\Gamma$-field by taking $\Gamma$ to be the graph of the exponential map. It is full if the field is algebraically closed and the exponential map is surjective. The reducts $\FEDE$ of differentially closed fields are full $\Gamma$-fields.

\subsection{Blurring an exponential field}

We can think of the set $\GDE$ of solutions to the exponential differential equation as a blurred version of the graph of the exponential function. Instead of $\class{(x,y)}{y=e^x}$ it is like $\class{(x,y)}{(\exists c \in \gm(C))y=c\cdot e^x}$, where $C$ is the field of constants. If $F$ is an abstract differential field rather than a field of functions this does not make literal sense. However it motivates the main idea of this paper, the notion of blurring the graph of exponentiation by a subgroup.
\begin{defn}
Let $\Fexp$ be an exponential field and let $\h$ be any subgroup of $\gm(F)$.  Let $\Gamma_\h$ be the subgroup
 \[\Gamma_\h = \class{(x,y)\in G(F)}{\frac{y}{\exp(x)} \in \h}\]
 of $G(F)$. Define $F_\h$ to be the $\Gamma$-field $\tuple{F;+,\cdot,\Gamma_\h}$.
\end{defn}

Taking $\h = \{1\}$, so no blurring at all, we get $\Fexp$ back. Taking $\h = \gm(F)$ we get $\Gamma_\h = G(F)$, the whole algebraic group. The blurred exponentiation $\GBE$ is the case where $\h = \gm(C)$ for the countable subfield $C = \ecl^F(\emptyset)$ and the approximate exponentiation $\GAE$ is the case when $\h = \exp(\Q + 2\pi i \Q)$.

\subsection{The pregeometry on a $\Gamma$-field}\label{predim section}

The Ax-Schanuel property for a $\Gamma$-field asserts the non-negativity of a certain predimension function. This leads to a pregeometry in the usual way.

Let $F_\Gamma$ be a full $\Gamma$-field, and let $A$ be a $\Gamma$-subfield of $F$, that is, $A$ is a subfield of $F$ and we define $\Gamma(A) \leteq G(A) \cap \Gamma$. We require that $\Gamma(A)$ is pure as a subgroup of $\Gamma$, that is, if $a \in \Gamma$ and $m \in \N^+$ are such that $ma \in \Gamma(A)$ then $a \in \Gamma(A)$. So in particular $\Gamma(A)$ is divisible and contains any torsion points which $\Gamma$ has.

 Let $b$ be a finite tuple from $\Gamma(F)$. Let $\gen{\Gamma(A),b}$ be the  divisible hull of $\Gamma(A) \cup b$ in $\Gamma$. Then the quotient group $\gen{\Gamma(A),b}/\Gamma(A)$ is a $\Q$-vector space, and we write the dimension of this space as $\ldim_\Q(b/\Gamma(A))$.
\begin{defn}
The \emph{predimension} of $b$ over $A$ is 
\[\delta(b/A) \leteq \td(b/A) - \ldim_\Q(b/\Gamma(A)).\]
\end{defn}

\begin{defn}
A $\Gamma$-subfield $A$ of $F$ is said to be \emph{(relatively) $\Gamma$-closed in $F$} if for every finite tuple $b$ from $\Gamma$ such that $\delta(b/A) \le 0$ we have $b \in \Gamma(A)$.
\end{defn}

\begin{fact}[{\cite[Proposition~4.14 and Lemma~4.10]{PEM}}]
The relatively $\Gamma$-closed subfields of a full $\Gamma$-field $F$ are the closed sets of a pregeometry on $F$. They are all full $\Gamma$-subfields of $F$.
\end{fact}
We write this pregeometry as $\Gcl^F$, or just as $\Gcl$ if the field $F$ is understood. For an exponential field $\Fexp$ we write the same pregeometry as $\ecl^F$ or as $\ecl$.

It is not immediate from the definition that a given $\Gamma$-field should have proper relatively $\Gamma$-closed subfields. It is true for the reducts $\FEDE$ by the Ax-Schanuel property, and for $\Bexp$ by construction. For $\Cexp$ it also follows from the Ax-Schanuel theorem via its role in the proof of the countable closure property.
\begin{defn}
A $\Gamma$-field $F_\Gamma$ has the \emph{countable closure property} (CCP) if whenenver $A \subs F$ is finite then $\Gcl(A)$ is countable.
\end{defn}

\begin{fact}\cite[Lemma~5.12]{Zilber05peACF0}\label{CCP theorem}
$\Cexp$ has the countable closure property.
\end{fact}

\begin{prop}\label{blurred pregeom}
Let $\Fexp$ be an exponential field and $\h$ a subgroup of $\gm(F)$. Let $\Gcl$ be the $\Gamma$-closure pregeometry on $F_\h$. Then for any subset $S \subs F$ we have
$\Gcl(S) = \ecl^F(S \cup \h)$. In particular, if $\h \subs \ecl^F(\emptyset)$ we have $\Gcl = \ecl^F$.
\end{prop}
\begin{proof}
First we show that $\h \subs \Gcl(\emptyset)$. Suppose $A$ is a $\Gamma$-subfield of $F$, and $h \in \h$. Then $(0,h) \in \Gamma_\h$. If $(0,h) \notin \Gamma_\h(A)$ then $\ldim_\Q((0,h)/\Gamma(A)) = 1$. But $\td((0,h)/A) \le 1$, so $\delta((0,h)/A) \le 1-1=0$. So $h \in \Gcl(A)$. So $h$ is in every relatively $\Gamma$-closed subfield of $F$, so in particular in $\Gcl(\emptyset)$. Thus $\h \subs \Gcl(\emptyset)$.

Now we will show that the $\Gcl$-closed subsets of $F$ are the same as the $\ecl$-closed subsets which contain $\h$. We write $\G$ for the graph of $\exp_F$. Suppose that $A$ is $\ecl$-closed in $F$, and $\h \subs A$. Let $b \in \Gamma_\h^n$ and let $c \in (\{0\} \cross \h)^n$ such that $b-c \in \G^n$, writing the group operation on $\G^n$ additively. Write $\delta_\h$ for the predimension in the sense of $F_\h$, and $\delta_{\exp}$ for the predimension in the sense of $\Fexp$. Then we have
\begin{eqnarray*}
\delta_\h(b/A) & = & \td(b/A) - \ldim_\Q(b/\Gamma_\h(A))\\
  		& = & \td(b-c/A) - \ldim_\Q(b-c/\Gamma_\h(A))\\
  		& = & \td(b-c/A) - \ldim_\Q(b-c/\G(A))\\
		& = & \delta_{\exp}(b-c/A)
\end{eqnarray*}
Thus if $\delta_\h(b/A) \le 0$ then $\delta_{\exp}(b-c/A) \le 0$ so, since $A$ is $\ecl$-closed, $b-c \in A$ and hence $b \in A$. So $A$ is $\Gcl$-closed. Conversely if $A$ is $\Gcl$-closed and $b \in \G^n$ with $\delta_{\exp}(b/A) \le 0$ then the above calculation with $c=0$ shows that $\delta_\h(b/A) = \delta_{\exp}(b/A) \le 0$, so $b \in A$. So $A$ is $\ecl$-closed.
\end{proof}

\begin{cor}\label{CCP cor}
If $\Fexp$ has the CCP and \h\ is countable then $F_\h$ has the CCP. \qed
\end{cor}

\begin{lemma}\label{blurred Ax-Schanuel}
Let $\Fexp$ be any exponential field, let $C$ be any relatively $\ecl$-closed subfield of $\Fexp$, and let $\GBE$ be the graph of exponentiation blurred by $\gm(C)$. Then $\tuple{F;+,\cdot,\GBE}$ satisfies
  the Ax-Schanuel property.
\end{lemma}
\begin{proof}
Suppose that $a \in \GBE^n$ satisfies $\td(a/C) < n+1$.
  Let $J$ be the smallest algebraic subgroup of $\gm^n$ such that $a \in TJ + G^n(C)$. If $a \in G^n(C)$ then we can take $J = \{1\}$ and we are done. So assume not.
  
By the definition of $\GBE$, $a = (b,c e^b)$ for some $b \in F^n$ and $c \in \gm(C)^n$, so $\td(b,e^b/C) = \td(a/C) < n+1$ and $(b,e^b)$ and $a$ lie in $C$-cosets of the same $TJ$. Since $C$ is relatively $\ecl$-closed in $\Fexp$ we have $1 \le \delta(b,e^b/C) = \td(b,e^b/C) - \ldim_\Q(b/C)$, and so $\ldim_\Q(b/C) < n$. Thus $J$ is a proper algebraic subgroup of $\gm^n$. So the Ax-Schanuel property holds for  $\tuple{F;+,\cdot,\GBE}$.
\end{proof}

\section{$\Gamma$-closedness axioms}\label{G-closed section}
 
There are a number of slight variants of the $\Gamma$-closedness axiom. The version \emph{full $\Gamma$-closedness} given in Fact~\ref{TEDE axioms} is relatively simple and, importantly, is first-order expressible. However it is useful to make further restrictions on the varieties $V$ which occur in it. For any $a \in F$ the subvariety $V_a$ of $G^1$ given by $x=a$ is rotund, since it has dimension 1. If $F_\Gamma$ is an algebraically closed $\Gamma$-field which satisfies full $\Gamma$-closedness then for any $a \in F$ there is $b \in \gm(F)$ such that $(a,b) \in \Gamma \cap V_a$. Thus $\pi_1(\Gamma) = F$. Similarly $\pi_2(\Gamma) = \gm(F)$, so $F_\Gamma$ is necessarily a full $\Gamma$-field. We can axiomatize fullness directly, and it is useful to exclude these subvarieties from the statement of the $\Gamma$-closedness axiom. More generally we can restrict to \emph{free} subvarieties, and indeed only those of exactly the critical dimension $n$.
\begin{defn}\label{free}
Let $F$ be an algebraically closed field and $V$ an irreducible subvariety of $G^n$, defined over $F$. Then $V$ is \emph{additively free} if $\pi_1(V) \subs \ga^n$ is not contained in any subvariety of $\ga^n$ defined by an equation of the form $\sum_{i=1}^n m_i x_i = c$, for any $m_i \in \Z$, not all zero, and any $c \in F$. $V$ is \emph{multiplicatively free} if $\pi_2(V) \subs \gm^n$ is not contained in any subvariety of $\gm^n$ defined by an equation of the form $\prod_{i=1}^n y_i^{m_i} = c$, for any $m_i \in \Z$, not all zero, and any $c \in F$. $V$ is \emph{free} if it is both additively free and multiplicatively free.
\end{defn}
\begin{fact}
If $F_\Gamma$ is a full $\Gamma$-field then it satisfies full $\Gamma$-closedness if and only if it satisfies
 \begin{description}
  \item[\textup{$\Gamma$-closedness}] For each $n \in
    \N$, and each free and rotund irreducible algebraic subvariety $V$ of
    $G^n$ of dimension $n$, the intersection $\Gamma^n \cap V$ is
    nonempty.
  \end{description}
\end{fact}
\begin{proof}
See \cite[Proposition~2.33]{TEDESV}, where it is even shown that further restrictions on $V$ can be given (perfect rotundity) while retaining an equivalent form of the axiom.
\end{proof}

The axiomatization of $\Bexp$ uses a stronger version of $\Gamma$-closedness, asking not just that $\Gamma^n \cap V$ is non-empty but for the existence
 of points which are suitably generic.
\begin{fact}
Up to isomorphism there is exactly one model of the following axioms of each uncountable cardinality, and all the models are quasiminimal.
  \begin{description}
  \item[\textup{Full $\Gamma$-field}] $F_\Gamma$ is a full $\Gamma$-field. 
  \item[\textup{Standard fibres}] $(0,y) \in \Gamma \iff y=1$, and there is a transcendental $\tau$ such that $(x,1) \in \Gamma \iff x \in \tau\Z$.
  \item[Schanuel property] If $a \in \Gamma^n$ satisfies $\td(a/\Q) < n$ then
    there is a proper algebraic subgroup $J$ of $\gm^n$ such that $a \in TJ$.
  \item[\textup{Strong $\Gamma$-closedness}] 
  For each $n \in  \N$, and each free and rotund irreducible algebraic subvariety $V$ of $G^n$, of dimension $n$, and each finite tuple $a \in \Gamma^n$, there is  $ b \in \Gamma^n \cap V$ such that $b$ is $\Q$-linearly independent over $a$.
\item[CCP] $F_\Gamma$ has the countable closure property.
  \end{description}
\end{fact}
\begin{proof}
Apart from a straightforward translation of the axioms into the language of $\Gamma$-fields, this is the main theorem of \cite{Zilber05peACF0}. A more complete account of the proof is given in \cite[Theorem~9.1]{PEM}.
\end{proof}
\begin{defn}\label{B defn}
The model of the above axioms of cardinality continuum is the exponential field denoted $\Bexp$.
\end{defn}

The exponential field $\Bexp$, or at least the model of the axioms of dimension $\aleph_0$, is constructed using the \Fraisse-Hrushovski amalgamation-with-predimension technique. The strong $\Gamma$-closedness axiom directly captures a \emph{richness} property of the \Fraisse\ limit, that every finitely generated $\Gamma$-field extension $B$ of a finitely generated $\Gamma$-subfield $A$ of $\Bexp$ which is \emph{strong} (preserves the Schanuel property) but of finite rank must be realised inside $\Bexp$. As it stands, it appears not to be first-order expressible. However, if the Zilber-Pink conjecture for tori is true then strong $\Gamma$-closedness is actually equivalent to $\Gamma$-closedness (given the other axioms) \cite[Theorem~5.7]{ECFCIT}.

In Section~2 of the paper \cite{TEDESV}, the countable saturated model of $\TEDE$ is also obtained by a Hrushovski-Fra\"iss\'e amalgamation-with-predimension construction. Formulating the analogous richness property in the differential equation setting is more tricky because of the constant field and because the Ax-Schanuel property is more complicated than the Schanuel property. Indeed, the formulation (SEC) in \cite[Definition~2.30]{TEDESV} is not completely correct and we take this opportunity to correct it. The formulation given here is from \cite[Definition~11.1]{PEM}, and the fact that it captures the richness property is Proposition~11.2 there.
\begin{defn}
Let $F_\Gamma$ be a full $\Gamma$-field and $C$ a full $\Gamma$-subfield of $F$. Then an \emph{Ax-Schanuel pair over $C$} is a pair $(V,a)$ where $a$ is a finite tuple from $\Gamma(F)$ which is $\Q$-linearly independent over $\Gamma(C)$ and $V \subs G^n$ is a free and rotund, irreducible subvariety of $G^n$ of dimension $n$ defined over $C(a)$ such that for $b \in V$, generic over $C(a)$, the locus $W \leteq \loc(a,b/C)$ is strongly rotund.

A subvariety $V \subs G^n$ is \emph{Ax-Schanuel good over $C$} if there is $a \in \Gamma^n$ such that $(V,a)$ is an Ax-Schanuel pair over $C$.

Using these definitions we can state two further variants of $\Gamma$-closedness axioms.
\begin{description}
  \item[\textup{Generic Strong $\Gamma$-closedness over $C$}] For each Ax-Schanuel pair $(V,a)$ over $C$, there is $b \in \Gamma^n \cap V$ which is $\Q$-linearly independent over $\Gamma(C) \cup a$.
  \item[\textup{Generic $\Gamma$-closedness over $C$}] For each $V$ which is Ax-Schanuel good over $C$, the intersection $\Gamma^n \cap V$ is non-empty.
  \end{description}
\end{defn}
It is not obvious that either generic strong $\Gamma$-closedness over $C$ or generic $\Gamma$-closedness over $C$ are first-order expressible. On the face of it, they appear to state some saturation property. However, an important breakthrough from \cite{PEM} shows that they are equivalent to each other and it follows that in the context of $\TEDE$ they are equivalent to $\Gamma$-closedness. 
\begin{fact}\label{G-closed equivalences}
If $F_\Gamma$ is a full $\Gamma$-field which is $\Gamma$-closed and $C$ is a relatively $\Gamma$-closed subfield then $F_\Gamma$ is generically strongly $\Gamma$-closed over $C$.
\end{fact}
\begin{proof}
By Proposition~11.5 of \cite{PEM}, generic strong $\Gamma$-closedness over $C$ is equivalent to generic $\Gamma$-closedness over $C$. The latter is an immediate consequence of $\Gamma$-closedness.
\end{proof}

The importance of this fact for us is that it has the following corollary.
\begin{fact}[{\cite[Corollary~11.7]{PEM}}]\label{G-closed + CCP implies qm}
If $F$ is a full $\Gamma$-field which is $\Gamma$-closed and has the countable closure property then $F$ is quasiminimal.
\end{fact}

\section{Quasiminimality and categoricity}\label{QM section}

The amalgamation-with-predimension method is used to produce the countable models of \TEDE\ and of the \Bexp-axioms. The uncountable quasiminimal models are then isolated using the method of quasiminimal pregeometry classes (also called quasiminimal excellent classes, because Shelah's excellence technique is what is used to build the uncountable models). Putting together results from \cite{TEDESV}, \cite{PEM}, and \cite{OQMEC} we can give a simple axiomatization for the quasiminimal models of \TEDE.
\begin{defn}
Let $\TEDE^*$ be the list of axioms consisting of the axioms for $\TEDE$  from Fact~\ref{TEDE axioms} except for non-triviality, together with
\begin{description}
  \item[CCP] $F_\Gamma$ has the countable closure property.
  \item[Constant field] $\td(C/\Q)$ is infinite.
  \end{description}
\end{defn}

\begin{prop}\label{qme prop} \
\begin{enumerate}
\item The class of models of $\TEDE^*$ is a quasiminimal pregeometry class with pregeometry $\Gcl$. 

\item Up to isomorphism there is exactly one model of each uncountable cardinality.

\item All models of $\TEDE^*$ are quasiminimal.

\item The non-trivial models of $\TEDE^*$ can also be defined as the prime models of $\TEDE$ over $C \cup B$ where $C$ is the fixed constant field with $\td(C/\Q) = \aleph_0$ and $B$ is an independent set of realisations of the unique generic type (the unique global 1-type of Morley rank $\omega$).
\end{enumerate}
\end{prop}
\begin{proof}
Let $M$ be the countable saturated model of $\TEDE$. By \cite[Theorem~2.35]{TEDESV}, $M$ is the unique countable model of $\TEDE$ which is infinite dimensional with respect to $\Gcl$, has $\td(C/\Q) = \aleph_0$ and satisfies the richness property of the \Fraisse\ limit, that is, it is generically strongly $\Gamma$-closed over $C$.

By Fact~\ref{G-closed equivalences}, $\Gamma$-closedness implies generic strong $\Gamma$-closedness over $C$. Thus $M$ is the unique countable model of $\TEDE$ which is infinite dimensional and has the fixed constant field $C$ of transcendence degree $\aleph_0$.

The model $M$ is also described in Section~9.7 of \cite{PEM}, where it is written as $M_{\mathrm{\Gamma\text{-}tr}}(C)$. By Theorem~6.9 of that paper, $M$ is a quasiminimal pregeometry structure. Let $\mathcal K(M)$ be the quasiminimal pregeometry class generated by $M$, that is, the smallest class of structures containing $M$ and all its $\Gcl$-closed substructures which is closed under isomorphism and under taking unions of directed systems of closed embeddings. By \cite[Theorem~4.2]{OQMEC}, in any quasiminimal pregeometry class the models are determined up to isomorphism by their dimension, and by CCP the cardinality of a model $M$ is equal to $\dim M + \aleph_0$.  The countable models in $\mathcal K(M)$ are precisely the countable models of $\TEDE$ minus non-triviality, such that $\td(C/\Q)$ is infinite. These axioms are first-order except the last which is $\Loo$-expressible, so together they are $\Loo$-expressible. Then by \cite[Theorem~5.5]{OQMEC}, the class $\mathcal{K}(M)$ is axiomatized by these axioms together with CCP.
 This proves (1) and (2). Item (3) is \cite[Lemma~5.1]{OQMEC}.

For (4), since $\TEDE$ is complete and $\aleph_0$-stable there is a prime model over $C \cup B$. Since there is a model $M_{|B|}$ of $\TEDE$ containing $C \cup B$ and satisfying CCP, the prime model must satisfy CCP so it can embed into $M_{|B|}$. Thus the prime model is a model of $\TEDE^*$. Since there is only one model of $\TEDE^*$ of dimension $|B|$, all non-trivial models of $\TEDE^*$ are of this form. 
\end{proof}
\begin{remark}
The trivial model of $\TEDE^*$ is the unique model of dimension 0 which is just the field $C$ with $\Gamma(C) = G(C)$.
\end{remark}

\section{The use of density to prove $\Gamma$-closedness}\label{density section}

We now prove the $\Gamma$-closedness property for $\CBE$ and $\CAE$, and in fact in a little more generality: for $\C_\h$ where $\h$ is a dense subgroup of $\gm(\C)$. As well as density with respect to the complex topology, the key tools are the Ax-Schanuel property and the fibre dimension theorem, which together are used to ensure that certain intersections are transversal and have the required dimension. See for example Section~2 of \cite{BMZ07} for a clear discussion of the fibre dimension theorem.

\begin{lemma}\label{rotund fibre lemma}
Suppose $V \subs G^n$ is irreducible and rotund, and $\dim V = n$. Let $J$ be a connected algebraic subgroup of $\gm^n$ given by a matrix equation $y^M = 1$. Let $\gamma \in V(\C)$. Then there is a Zariski-open subset $V_J$ of $V$ such that if $\gamma \in V_J$ then $\dim (V \cap \gamma{+}TJ) \le \dim J$.
\end{lemma}
\begin{proof}
The intersection $V \cap \gamma{+}TJ$ is a fibre of the regular map $V \to M\cdot V$. So by the fibre dimension theorem there is a Zariski-open $V_J$ such that for $\gamma \in V_J$ we have $\dim (V \cap \gamma{+}TJ) = n - \dim M\cdot V$. Since $V$ is rotund, $\dim M\cdot V \ge \rk M = n-\dim J$. Thus for $\gamma \in V_J$ we have $\dim (V \cap \gamma{+}TJ) \le \dim J$ as required. 
\end{proof}

For a subgroup $\h$ of $\gm(\C)$, recall that $\Gamma_\h = \class{(x,y)\in G(F)}{\frac{y}{\exp(x)} \in \h}$ and $\C_\h = \tuple{\C;+,\cdot,\Gamma_\h}$.

\begin{prop}\label{density implies G-closed}
  If $\h$ is any subgroup of $\gm(\C)$ which is dense in the complex
  topology then $\C_\h$ is $\Gamma$-closed.
\end{prop}
\begin{proof}\footnote{Thanks to Margaret Friedland who gave a helpful answer on MathOverflow to help fill a gap in an earlier version of this proof. \url{https://mathoverflow.net/q/343181} } 
  Let $V$ be an irreducible, rotund subvariety of $G^n(\C)$, of
  dimension $n$. Let $a$ be a regular point of $V(\C)$, and let $U$ be
  a small neighbourhood of $a$ in $V(\C)$, with respect to the complex topology. We may take $U$ to be
  analytically diffeomorphic to an open ball in $\C^n$ because $a$ is a regular point. 
  Let $\theta: G^n(\C) \to \gm^n(\C)$ be the map given by 
  \[\theta(x_1,\ldots,x_n,y_1,\ldots,y_n) = \left(\frac{y_1}{\exp(x_1)},\ldots,\frac{y_n}{\exp(x_n)}\right).\]   
  
  Let $t = \theta(a) \in \gm^n(\C)$ and let $A = \theta^{-1}(t) \cap U$. Then $A$
  is an analytic subset of $U$, so, by taking $U$ sufficiently small,
  we may assume that $A$ is connected. Thus $A$ is a connected neighbourhood of $a$ in the intersection $a{+}\G^n \cap V(\C) \subs G^n(\C)$.

  Suppose that $\dim A = 0$. Then, since it is connected, $A$ is the singleton $\{a\}$, and indeed $\{a\}$ is an isolated point of $\theta^{-1}(t)$.
 So, shrinking $U$ if necessary, $\theta\restrict{U}$ is finite-to-one \cite[p257]{Lojasiewicz}. Then, by the Remmert open mapping theorem \cite[p297]{Lojasiewicz}, $\theta(U)$ is open in $\gm^n(\C)$.
  Now $\h^n$ is dense in $\gm^n(\C)$ by assumption, so there is $c \in \theta(U) \cap \h^n$. 
  Then $\theta^{-1}(c) \in \Gamma_\h^n \cap U \subs \Gamma_\h^n \cap V$. 
  Thus $\Gamma_\h^n \cap V$ is nonempty as required.

\medskip
 So it remains to show that we can find a regular point $a \in V(\C)$ such that $\dim A = 0$. (Although $A$ depends on the choice of $a$ and the choice of neighbourhood $U$, we will suppress these dependences from the notation.)

Suppose that $\dim A = d > 0$.
Let $F$ be the field of germs at $a$ of complex analytic functions $A \to \C$. In $F$ we have the restrictions to $A$ of the coordinate functions $x_i$, $y_i$ on $G^n$, where the $x_i$ are the coordinates from $\ga^n$ and the $y_i$ are the coordinates from $\gm^n$. Since $\dim A = d$, there is a $d$-dimensional vector space $\Delta$ of $\C$-derivations on $F$. Since $A \subs \theta^{-1}(t)$, we have $y_i = t_i\exp(x_i)$ for each $i$. So for each $i=1,\ldots,n$ and each derivation $D \in \Delta$ we have $D y_i = y_i D x_i$. Let $r = \ldim_\Q(x_1,\ldots,x_n/\C)$ so we can choose $M \in \Mat_n(\Z)$, a matrix of rank $n-r$, and $\gamma \in G^n(\C)$  such that $(Mx,y^M) = \gamma$. Write $J$ for the connected component of the algebraic subgroup of $\gm^n$ given by $y^M = 1$ and $TJ$ for the connected component of the algebraic subgroup of $G^n$ given by $Mx = 0$ and $y^M = 1$. So $\dim J = r$ and $(x,y) \in \gamma{+}TJ$.

Then $(x,y)$ is an $F$-point of the algebraic variety $V \cap \gamma{+}TJ$, which is defined over $\C$, so $\td(x,y/\C) \le \dim V \cap  \gamma{+}TJ$.
By Ax's theorem~\cite[Theorem 3]{Ax71} we have $\td(x,y/\C) - r \ge d$ and so 
\[\dim (V \cap  \gamma{+}TJ) \ge d+ \dim J > \dim J.\]
 By Lemma~\ref{rotund fibre lemma}, this can only occur if $v \notin V_J$.

Let $V^o =  V^{\mathrm{reg}}(\C) \cap \bigcap_{J \le \gm^n} V_J(\C)$, where $V^{\mathrm{reg}}$ is the set of regular points in $V$, which is Zariski-open in $V$. Each $V \minus V_J$ is a Zariski-closed subset of $V$ of lower dimension, and there are only countably many algebraic subgroups $J$, so $V^o$ is non-empty. Taking any $v \in V^o$ we have that $\dim A = 0$, which completes the proof.
\end{proof}

\begin{remark} 
We could use the uniform version of Ax's theorem \cite[Theorem~4.3]{TEDESV} to see that only finitely many groups $J$ need to be considered, so $V^o$ is in fact Zariski-open in $V$. However we do not need that for the proof.
\end{remark}

\section{Proofs of the main theorems}\label{proofs section}

In this section we bring together the results of the paper to prove the main theorems.
\begin{prop}\label{models of TEDE*}
  The blurred exponential fields $\CBE$ and $\BBE$ are models of the theory $\TEDE^*$.
\end{prop}
\begin{proof}
  The axioms $\ACF_0$, Group, and Fibres are immediate from the
  definition. The Ax-Schanuel property is given by Lemma~\ref{blurred Ax-Schanuel}.
The CCP holds by Corollary~\ref{CCP cor}. 

$\CBE$ and $\BBE$ are full $\Gamma$-fields because the exponential maps are surjective. $\CBE$ is $\Gamma$-closed by Proposition~\ref{density implies G-closed}. To show that $\BBE$ is $\Gamma$-closed, let $V \subs G^n$ be rotund and irreducible. Since $\Bexp$ is strongly $\Gamma$-closed, $\G^n \cap V \neq \emptyset$.  But $\G \subs \GBE$, so $\GBE^n \cap V \neq \emptyset$. 

It remains to show that $\td(C/\Q)$ is infinite. For $\B$, this follows at once from the Schanuel property since $C$ contains all the numbers $\exp(1)$, $\exp(\exp(1))$, $\exp(\exp(\exp(1))),\ldots$, which are algebraically independent.

For $\C$ we cannot use the Schanuel property, but  the Lindemann-Weierstrass theorem is enough. This is the
  special case of the Schanuel property for algebraic numbers and their exponentials. Since
  the field of algebraic numbers is infinite dimensional as a
  $\Q$-vector space, its image under $\exp$ must have infinite
  transcendence degree.
\end{proof}

\begin{proof}[Proof of Theorem~\ref{blurred exp theorem}] \
\begin{enumerate}
\item Proposition~\ref{models of TEDE*} shows that $\BBE$ and $\CBE$ both satisfy the theory $\TEDE^*$.
They are both of cardinality $2^{\aleph_0}$ so, by  Proposition~\ref{qme prop} items~(2) and~(3), they are isomorphic to each other and quasiminimal.

\item Let $\Fdiff$ be the prime model of $\DCF_0$ over $C \cup B$ where $C$ is a constant field of transcendence degree $\aleph_0$ and $B$ is a differentially independent set of size continuum. Then in the reduct $\FEDE$, $B$ is an independent set in the generic type and $\FEDE$ is a prime model of $\TEDE$ over $C \cup B$. So by Proposition~\ref{qme prop}, $\FEDE$ is a model of $\TEDE^*$, and hence is isomorphic to $\CBE$ and $\BBE$.

\item Since $\DCF_0$ is $\aleph_0$-stable, so is its reduct $\FEDE$.

\item As discussed before, the unique model of cardinality continuum of $\TEDE^*$ is constructed by the amalgamation-with-predimension method (for the countable model) and then the excellence method (to build uncountable models as directed limits of countable models) in Section~9.7 of \cite{PEM}, where it is written as $\mathbb M_{\mathrm{\Gamma\text{-}tr}}(C)$.
\end{enumerate}
\end{proof}

\begin{proof}[Proof of Theorem~\ref{pregeom theorem}]
By Proposition~\ref{blurred pregeom}, the pregeometries $\ecl^\C$ and $\Gcl^{\CBE}$ are equal, as are $\ecl^\B$ and $\Gcl^{\BBE}$. Theorem~\ref{pregeom theorem} then follows from Theorem~\ref{blurred exp theorem}.
\end{proof}

We prove Theorem~\ref{CAE is qm} in slightly more generality, since the only relevant properties of the group $\h = \exp(\Q + 2\pi i \Q)$ are that it is countable and dense.
\begin{theorem}
Let $\h \subs \gm(\C)$ be a countable subgroup which is dense in the complex topology. Then the structure $\C_\h \leteq \tuple{\C;+,\cdot,\Gamma_\h}$ is quasiminimal. In particular, $\CAE$ is quasiminimal.
\end{theorem}
\begin{proof}
By Theorem~\ref{CCP theorem}, $\Cexp$ has the countable closure property, and by Corollary~\ref{CCP cor} so does $\C_\h$.
By Proposition~\ref{density implies G-closed}, $\C_\h$ is $\Gamma$-closed. Thus by Fact~\ref{G-closed + CCP implies qm}, $\C_\h$ is quasiminimal.
\end{proof}

\section{Final remarks}\label{remarks section}

We finish with some remarks about what the results of this paper say towards proving quasiminimality for $\Cexp$ itself and related structures.

First observe that $\CAE$ is a reduct of $\Cexp$. For $\Q$ is defined in $\Cexp$ by the formula 
\[\exists y_1 \exists y_2 [e^{y_1} = 1 \wedge e^{y_2} = 1 \wedge x \cdot y_1 = y_2 \wedge y_1 \neq 0]\]
and $2 \pi i \Q$ is defined (without the parameter for $2\pi i$) by
\[\exists y z[z \in \Q \wedge e^y = 1 \wedge x = y\cdot z].\]
So $\Q + 2 \pi i \Q$ and $\GAE$ are $\emptyset$-definable in \Cexp.

On the other hand $\Q + 2 \pi i\Q$ is definable in $\CAE$ by $(x,1) \in \GAE$, and $\Q$ is definable as the multiplicative stabilizer of $\Q + 2 \pi i\Q$. So whereas $\CBE$ is $\aleph_0$-stable, $\CAE$ interprets arithmetic, like $\Cexp$. So in stability-theoretic terms, $\CAE$ appears to be close to \Cexp. I cannot prove that it is a proper reduct of \Cexp, but it seems likely because the analogous reduct $\BAE$ of \Bexp\ is a proper reduct. 
To see this, consider the $\Gamma$-subfield $A \subs \BAE$, where as a field $A = \Q(2\pi i,\exp(\Q),\sqrt1)$ and $\Gamma(A) = (\Q + 2 \pi i\Q) \times (\exp(\Q) \cdot \sqrt1)$. By $2 \pi i$ in $\B$ we just mean a cyclic generator of the kernel of the exponential map. Then $A$ is of finite rank as a $\Gamma$-field and is strongly embedded in $\BAE$, so every automorphism of $A$ extends to an automorphism of $\BAE$. Since $2 \pi i$ and $\exp_\B(1)$ are algebraically independent, there is an automorphism of $A$ fixing $\Q(\exp(\Q),\sqrt1)$ and sending $2 \pi i$ to $2 \pi i q$ for any given $q \in \Q \minus \{0\}$. However in $\Bexp$ the pair $\pm 2 \pi i$ is defined by
\[e^x=1 \wedge \left(\forall y \right)
\left(\exists n\in \Z\right)
\left[e^y = 1 \to nx=y\right]\]
where $\Z$ is defined by 
\[\forall y [e^y=1 \to e^{xy} =1].\]
So $\BAE$ has more automorphisms than $\Bexp$, hence it is a proper reduct.

An obvious question to ask is whether $\C_\h$ can be shown to be quasiminimal when $\h$ is a smaller group than $\exp(\Q + 2 \pi i \Q)$, for example $\exp(\Q)$ or the torsion group $\gm(\mathrm{tors}) = \exp(2 \pi i \Q)$. Of course taking $\h=\{1\}$ would give $\Cexp$ itself. The arguments of this paper immediately break down, because $\exp(\Q + 2 \pi i \Q)$ has no proper subgroups which are both divisible and dense in $\gm(\C)$.

\medskip 

It seems likely that the methods of this paper can also be used to prove quasiminimality for reducts of the complex field with a Weierstrass $\wp$-function. 
A complex elliptic curve $E$ has an exponential map of the form $\exp_E(z) = [\wp(z):\wp'(z):1]$. By analogy with $\GAE$ we can define a subgroup $\Gamma_{\mathrm A\wp} \subs \ga(\C) \cross E(\C)$ by $(x,y) \in \Gamma_{\mathrm A\wp}$ if $y-x \in E(\mathrm{tors})$.
We can also define a correspondence between two complex elliptic curves $\Gamma_{\mathrm{corr}} \subs E_1(\C) \cross E_2(\C)$ by $
\Gamma_{\mathrm{corr}} = \class{(\exp_{E_1}(z), \exp_{E_2}(z))}{z \in \C}$. I conjecture the following.
\begin{conj} \
 The structures $\C_{\mathrm A\wp} = \tuple{\C;+,\cdot,\Gamma_{\mathrm A\wp}}$ and
 $\C_{\mathrm{corr}} = \tuple{\C;+,\cdot,\Gamma_{\mathrm{corr}}}$ are quasiminimal.
\end{conj}
These conjectures would immediately follow from the generalization of Zilber's weak conjecture:

\begin{conj}
The complex field expanded by exponentiation and the maps $\exp_E$ for any countable set of complex elliptic curves $E$ is quasiminimal.
\end{conj}
One could generalize further and replace the elliptic curves by any commutative complex algebraic groups. Almost everything known towards Zilber's conjecture depends one way or another on the Ax-Schanuel theorem, which is known in this generality \cite{Ax72a}.


\end{document}